\documentclass{article}
\usepackage[latin9]{inputenc}
\usepackage{a4}
\usepackage{amsfonts}
\usepackage{amssymb}
\usepackage{amsmath}
\usepackage{amsthm}
\usepackage{changepage}
\usepackage{nomencl}
\usepackage{geometry}
\usepackage[pdftex]{graphicx}
\usepackage[toc,page]{appendix}
\begin{document}
\newtheorem{satz}{Theorem}
\newtheorem{lemma}[satz]{Lemma}
\newtheorem{prop}[satz]{Proposition}
\newtheorem{kor}[satz]{Corollary}
\theoremstyle{definition}
\newtheorem{defi}{Definition}

\newcommand{\cc}{{\mathbb{C}}}   
\newcommand{\ff}{{\mathbb{F}}}  
\newcommand{\nn}{{\mathbb{N}}}   
\newcommand{\qq}{{\mathbb{Q}}}  
\newcommand{\rr}{{\mathbb{R}}}   
\newcommand{\zz}{{\mathbb{Z}}}  
\newcommand{\In}{\subseteq}      
\newcommand{\sep}{\mathbin:}    
\newcommand{\Hl}{\operatorname{H}}
\author{Joachim K\"onig} 
\title{A family of polynomials with Galois group $PSL_5(2)$ over $\qq(t)$} 
\date{} 
\maketitle
{\abstract{We compute a family of coverings with four ramification points, defined over $\qq$, with regular Galois group $PSL_5(2)$.\\
On the one hand, this is (to my knowledge) the first explicit polynomial with group $PSL_5(2)$ over $\qq(t)$. On the other hand, it also positively answers the question whether $PSL_5(2)$ is the monodromy group of a rational function over $\qq$. At least this does not follow from considering class triples in $PSL_5(2)$, as there are no rigid, rational genus-zero triples. Also, for 4-tuples, our family is the only one with a Hurwitz curve of genus zero (see however below for the question whether this curve can be defined as a rational curve over $\qq$). There are also genus zero families with five branch points (cf. the table at the end of \cite{MSV}), and maybe their Hurwitz spaces can be shown to have rational points by arguments as in \cite{De}; however, so far I have not seen such arguments.}}
\\
\\
Let $G=PSL_5(2)$, and denote by $2A$ the class of involutions of cycle type $(2^8.1^{15})$,by $3B$ the class of elements of order 3 with cycle type $(3^{10}.1)$ in $G$, and by $8A$ the unique class of elements of order 8 in $G$ (of cycle type $(8^2.4^3.2.1)$). \\
We consider the set $SNi(C)$ ("straight Nielsen class", cf. e.g. \cite{DF2}) of class tuples of length 4, of type $(2A,2A,3B,8A)$ in $G=PSL(5,2)$, generating $G$ and having product $1$, i.e. \[SNi(C):=\{(\sigma_1,...,\sigma_4) \in G\mid \sigma_1,\sigma_2 \in 2A, \sigma_3\in 3B, \sigma_4 \in 8A, \langle \sigma_1,...,\sigma_4\rangle = G, \sigma_1 \cdots \sigma_4 = 1\}\]
The action of the braid group on $SNi(C)$ (cf. \cite{MM} or \cite{FV} for introductions to this topic) yields the following:\\
There is a family of covers $\mathcal{T} \mapsto \mathcal{C} \times \mathbb{P}^1\mathbb{C}$, where $\mathcal{C}$ ($C_2$-symmetrized reduced Hurwitz space) is an absolutely irreducible curve of genus zero and for every $h \in \mathcal{C}$ the corresponding fiber cover is a Galois cover of $\mathbb{P}^1\mathbb{C}$ with Galois group $PSL_5(2)$.
\\
\\
Although the usual braid genus criteria yield that the $C_2$-symmetrized Hurwitz space for this family is a genus-zero curve, it does not seem clear the standard theoretic considerations (e.g. odd cycle argument for the braid group generators, as in \cite[Chapter III. 7.4.]{MM}) whether it can also be defined as a rational curve over $\qq$. 
\\
In particular, the cycle structure of the braid orbit generators acting on the Nielsen class does not yield any places of odd degree. More precisely, the image of the braid group is imprimitive on then 24 points, with 12 blocks of length 2 (i.e. if $F|\qq(t)$ is the corresponding function field extension, of degree 24, we have an inclusion $\qq(t)\subset E \subset F$, with $[E:\qq(t)]=12$ and $[F:E]=2$). As the images in the action on the blocks of the three braids defining the ramification structure of these fields have cycle structure $(4^2.3.1)$, $(7.3.2)$ and $(2^5.1^2)$ respectively, it is clear that $E$ is still a rational function field; however the cycle structure of the latter involution in the action on 24 points is $(2^{12})$, so it is possible that a degree-2 place of $E$ ramifies in $F$, in which case the rationality of $F$ is not guaranteed.
\footnote{Closer group theoretic examination yields some evidence for prime divisors of odd degree: namely, the two $3$-cycles of the braid group generator of cycle structure $(7^2.3^2.2^2)$ correspond to degenerate covers with three ramification points, generating two isomorphic, but {\it non-conjugate} (in $PSL_5(2)$) subgroups. The same holds for the two $2$-cycles of this braid group generator. The explicit computations show, that the corresponding prime divisors of ramification index 3 and 2 respectively are indeed of degree 1.}
\\ \\
We therefore clarify the rationality of this curve by explicit computation.
We start with a degenerate cover with ramification structure $(2A,21A,8)$, with group $PSL_5(2)$. We solve the corresponding system of equations for the three-point cover modulo a suitable prime, and then lift and retrieve algebraic numbers from the $p$-adic expansions. The triple is rigid, but as the conjugacy class of the element of order 21 is not rational, we obtain a solution over a quadratic number field, namely
\[0=x^{21}\cdot(x-1)^7\cdot(x-a_1)^3-t\cdot(x^2-2\cdot x+a_2)^8\cdot(x^3-2\cdot x^2+a_3\cdot x+a_4)^4\cdot(x-a_5),\]
where 
$(a_1,...,a_5) := (\frac{1}{8}(-\sqrt{-7} + 11), \frac{1}{16}(-\sqrt{-7} + 11),\frac{1}{16}(\sqrt{-7} + 21), \frac{1}{128}(-3\sqrt{-7} - 31),\frac{1}{8}(-\sqrt{-7} + 3)
)$
\\
\\
From this degenerate cover, we developped a cover branched in four points, by complex approximations, using Puiseux expansions.\\
This is a non-trivial part of the computations; I will expand on this more closely in a forthcoming paper.
Similar techniques were outlined by Couveignes in \cite{Cou}.
\\
\\
Let $\mathbb{C}(x)|\mathbb{C}(t)$ be the corresponding field extension of rational function fields for the cover with four branch points. Via Moebius transformations (in $x$ and in $t$) it is possible to assume a defining polynomial
\[f:=f_0(x)^3\cdot (x-3)-t\cdot g_0(x)^8\cdot g_1(x)^4\cdot x,\] where $deg(f_0) = 10$, $deg(g_0)=2$ and $deg(g_1)=3$.
\\
(So we have e.g. assumed the element of order 8 to be the inertia group generator over infinity, and the element of order 3 the one over zero).\\
Also, assume that for some $\lambda \in \cc$ the polynomials
\[f_a:=f_0(x)^3\cdot (x-3)-a\cdot g_0(x)^8\cdot g_1(x)^4\cdot x\]
and \[f_{b}:=f_0(x)^3\cdot (x-3)-b\cdot g_0(x)^8\cdot g_1(x)^4\cdot x\]
(where $a$ and $b$ shall denote the complex zeroes of $x^2+x+\lambda$)
become inseparable in accordance to the elements in the conjugacy class $2A$ .\\
\\
Once we have obtained a complex approximation of such a polynomial $f$,
we now slowly move the coefficient at $x^2$ of the above polynomial $g_1$ to a fixed rational value, and apply Newton iteration to expand the other coefficients with sufficient precision to then retrieve them as algebraic numbers (using the LLL-algorithm). One finds that all the remaining coefficients come to lie in a cubic number field. This already indicates that there is a rational function field of index 3 in the (genus-zero) function field of the Hurwitz space, which would enforce the latter function field to be rational over $\qq$ as well. This will be confirmed by the remaining computations.\\
\\
We now choose a prime $p$ such that the above solution, found over a cubic number field, reduces to an $\mathbb{F}_p$-point.\\
Then we lift this point to sufficiently many $p$-adic solutions (all coalescing modulo $p$), in order to obtain algebraic dependencies between the coefficients\footnote{Alternatively, one could just repeat the process of rational specialization and Newton iteration, as above, sufficiently often, obtaining cubic minimal polynomials for the other coefficients in each case, and then interpolate.}. These dependencies are all of genus zero, and luckily some of them are of very small degree, e.g. if $c_2$ and $c_1$ are the coefficients at $x^2$ resp. $x$ of the polynomial $g_1$, one obtains an equation of degrees 2 and 3 respectively.\\ One can easily find a parameter $\alpha$ for the rational function field defined by this equation, using Riemann-Roch spaces. 
\\
\\
Now, we can express all the coefficients as rational functions in $\alpha$, and obtain the following result:
\begin{satz}
 Let $\alpha, t$ be algebraically independent transcendentals over $\qq$.\\
Define elements $a_1,...,a_{14}$ as follows:\\
\\
\begin{adjustwidth}{-2cm}{-2cm}
\begin{footnotesize}
\[a_1:=2\cdot \frac{\alpha^4 - 37/2\cdot\alpha^3 - 3\cdot\alpha^2 - 109/2\cdot\alpha + 91}{(\alpha-2)\cdot(\alpha+1)^2},\]
\[a_2:=-12\cdot\frac{(\alpha^2 - 5/2\cdot\alpha - 8)\cdot(\alpha^3 + 10\cdot\alpha^2 + 17\cdot\alpha + 44)}{(\alpha-2)\cdot(\alpha+1)\cdot(\alpha+4)},\]
\[a_3:=-8\cdot \frac{\alpha^6 - 93/2\cdot\alpha^5 - 261/2\cdot\alpha^4 - 1343/2\cdot\alpha^3 - 507/2\cdot\alpha^2 - 963\cdot\alpha + 3076}{(\alpha-2)\cdot(\alpha+1)\cdot(\alpha+4)},\]
\[a_4:=62\cdot\frac{\alpha^9 + 159/31\cdot\alpha^8 - 1410/31\cdot\alpha^7 - 3834/31\cdot\alpha^6 - 987\cdot\alpha^5 - 52905/31\cdot\alpha^4 - 72480/31\cdot\alpha^3 - 50196/31\cdot\alpha^2 + 256224/31\cdot\alpha - 56384/31}{(\alpha-2)^2\cdot(\alpha+1)\cdot(\alpha+4)^2},\]
\[a_5:=-36\cdot\frac{\alpha^9 + 241/6\cdot\alpha^8 + 226/3\cdot\alpha^7 + 521/3\cdot\alpha^6 + 1483/3\cdot\alpha^5 - 19319/6\cdot\alpha^4 - 12872/3\cdot\alpha^3 - 58186/3\cdot\alpha^2 - 51344/3\cdot\alpha + 30944/3}{(\alpha-2)^2\cdot(\alpha+4)^2},\]
\[a_6:=-92\cdot\frac{(\alpha + 1)^2\cdot(\alpha^7 - 337/23\cdot\alpha^6 - 237/23\cdot\alpha^5 - 2075/23\cdot\alpha^4 - 7724/23\cdot\alpha^3 + 2166/23\cdot\alpha^2 - 35740/23\cdot\alpha + 44048/23)}{(\alpha-2)^2\cdot(\alpha+4)},\]
\[a_7:=152\cdot\frac{(\alpha + 1)^3\cdot(\alpha^7 + 5/38\cdot\alpha^6 - 489/38\cdot\alpha^5 - 407/38\cdot\alpha^4 - 4367/38\cdot\alpha^3 - 3768/19\cdot\alpha^2 - 2882/19\cdot\alpha - 23192/19)}{(\alpha-2)^2\cdot(\alpha+4)},\]
\[a_8:=-87\cdot\frac{(\alpha + 1)^4\cdot(\alpha^6 + 96/29\cdot\alpha^5 + 66/29\cdot\alpha^4 + 532/29\cdot\alpha^3 - 3/29\cdot\alpha^2 + 300/29\cdot\alpha + 4724/29)}{(\alpha-2)^2},\]
\[a_9:=18\cdot\frac{(\alpha + 1)^6\cdot(\alpha^5 + 53/6\cdot\alpha^4 + 70/3\cdot\alpha^3 + 111/2\cdot\alpha^2 + 376/3\cdot\alpha + 94/3)}{(\alpha-2)^2},\]
\[a_{10}:=-6\cdot\frac{(\alpha+1)^9\cdot(\alpha+4)}{\alpha-2},\]
\[a_{11}:=-(\alpha + 1)^2,\]
\[a_{12}:=\frac{(\alpha + 1)\cdot(\alpha^2 - 16\cdot\alpha - 8)}{(\alpha-2)\cdot(\alpha+4)},\]
\[a_{13}:=-3(\alpha+1)^2,\]
\[a_{14}:=\frac{(\alpha+1)^3\cdot(\alpha+4)}{\alpha-2}.\]
\end{footnotesize}
\end{adjustwidth}
and set 
\[f_0:=x^{10} + a_1x^9 + ... + a_{10},\] 
\[g_0:=x^2-6x+a_{11},\]
\[g_1:=x^3+a_{12} x^2+a_{13} x+a_{14}.\]
Then the polynomial $f(x,\alpha,t):=f_0^3 \cdot (x-3)-t\cdot g_0^8\cdot g_1^4\cdot x$, of degree 31 in $x$, has Galois group $PSL_5(2)$ over $\qq(\alpha,t)$.
\end{satz}
\begin{proof}
 Dedekind reduction, together with the list of primitive groups of degree 31, shows that $PSL_5(2)$ must be a subgroup of the Galois group. It therefore suffices to exclude the possibilities $A_{31}$ and $S_{31}$.\\
Multiplying $t$ appropriately, we can assume the covers to be ramified in $t=0, t=\infty$ and the zeroes of $t^2+t+\lambda$, with some parameter $\lambda$. Interpolating through sufficiently many values of $\alpha$ one finds the degree-24 rational function $\lambda=\frac{h_1(\alpha)}{h_2(\alpha)}$ parametrizing the Hurwitz curve. As e.g. $\alpha=0$ and $\alpha=1/2$ yield the same value for $\lambda$, we set $t=C \cdot (\frac{f_0^3 \cdot (x-3)}{g_0^8\cdot g_1^4\cdot x})(s,0)$ (evaluating $x$ to a parameter $s$ of a a rational function field, as well as $\alpha$ to $0$, and multiplying with a suitable constant $C$ to obtain the above condition on the branch points). Then one can check that over $\qq(s)$, the polynomial $f(x,1/2, C_2\cdot t)$ (again for suitable constant $C_2$ to obtain the branch point conditions) splits into two factors of degrees 15 and 16. This means that for this particular point of the family, there is an index-31 subgroup of the Galois group that acts intransitively on the roots. As $PSL_5(2)$ has such a subgroup and $A_{31}$ and $S_{31}$ don't, the Galois group for this particular specialization is $PSL_5(2)$. This specialization corresponds to an unramified point on the (irreducible) Hurwitz spaces, therefore the entire family must belong to the same Hurwitz space and therefore have Galois group $PSL_5(2)$ over $\qq(\alpha,t)$.
\end{proof}
We can now specialize $\alpha$ to any value that does not let two or more ramification points coalesce, to obtain nice polynomials with nice coefficients with group $PSL_5(2)$ over $\qq(t)$.\\
E.g. $\alpha \mapsto 0$ leads to
\[\tilde{f}(x,t):=(x^5 - 95\cdot x^4 - 110\cdot x^3 - 150\cdot x^2 - 75\cdot x - 3)^3\cdot (x^5 + 4\cdot x^4 - 38\cdot x^3 + 56\cdot x^2 + 53\cdot x - 4)^3\cdot(x-3)\]
\[ - t\cdot(x^2 - 6 x - 1)^8\cdot (x^2-x-1)^4\cdot (x+2)^4\cdot x.\]
In fact it can be seen from $\lambda=\frac{h_1(\alpha)}{h_2(\alpha)}$ (as in the proof above) that the only specialized values for $\alpha$ that lead to degenerate covers (with two branch points coalescing) are $\alpha \mapsto -4$, $\alpha \mapsto -1$ and $\alpha \mapsto 2$.
\\
\\
{\bf Remark:}\\
The above proof esentially uses the fact that $PSL_5(2)$ has two arithmetically equivalent actions on 31 points. This can of course be applied to other linear groups, and has e.g. been used in \cite{Ma1} to verify $PSL_2(11)$ (and others) as the Galois group of a family of polynomials. 

\end{document}